\documentclass[11pt]{article}
\usepackage{multicol}
 \usepackage{amsfonts,amssymb,graphicx,amsthm,amsmath}
\usepackage[colorlinks=true,linkcolor=blue,citecolor=blue]{hyperref}

\newtheorem{rema}{Remark}

\newtheorem{lemm}{Lemma}
\newtheorem{theo}{Theorem}
\newtheorem{coro}{Corollary}

\newcommand{\M}{{\cal M}}

\newcommand{\R}[1][]{\ensuremath{{\mathbb{R}^{#1}} }}
\renewcommand{\S}[1][]{\ensuremath{{\mathbb{S}^{#1}} }}
\renewcommand{\H}[1][]{\ensuremath{{\mathbb{H}^{#1}} }}

\newcommand{\q}[1][]{\ensuremath{\mathbb Q^{n+1}_c}}

\newcommand{\s}{{\cal S}}
\newcommand{\<}{\langle}
\renewcommand{\>}{\rangle}
\newcommand{\ga}{\gamma}
\newcommand{\pa}{\partial}

\newcommand{\eps}{\varepsilon}
\newcommand{\te}{\theta}

\newcommand{\ka}{\kappa}

\newcommand{\co}{{\texttt{cos}_c}}
\newcommand{\si}{{\texttt{sin}_c}}

\newcommand{\ct}{{\texttt{cot}_c}}

\newcommand{\coca}{{\texttt{cos}^2_c}}
\newcommand{\sica}{{\texttt{sin}^2_c}}

\date{}

\title{Codimension two marginally trapped submanifolds in Robertson-Walker spacetimes}
\author{ Henri Anciaux\footnote{Universidade de S\~ao Paulo; supported by CNPq (PQ 306154/2011-0) and Fapesp (2011/21362-2)}, Nastassja Cipriani\footnote{KU Leuven; supported by Belgian Interuniversity Attraction Pole  P07/18 (Dygest)}}

\begin{document}

\maketitle

\centerline{\textbf {\large{Abstract}}}

\bigskip

{\small  We give a local characterization of codimension two  submanifolds which are marginally trapped in Robertson-Walker spaces, in terms of an algebraic equation to be satisfied by the height function. We prove the existence of a large number of local solutions. We refine the description in the case of curves with lightlike acceleration in three-dimensional spaces Robertson-Walker spaces, and in the case of codimension two submanifolds whose second fundamental form is lightlike.}

\bigskip

\centerline{\small \em 2010 MSC: 53A10,  53C42
\em }

\bigskip

\section*{Introduction}
 A spacelike surface of a spacetime is said to be \em trapped \em if its mean curvature vector is timelike. This concept  was first introduced in \cite{Pe} in the study of singularities of spacetime. %We refer the reader to \cite{Se} for a modern exposition. 
The limiting case of \em marginally trapped surfaces, \em i.e.\ surfaces whose mean curvature vector is lightlike, has recently attracted the attention of both  physicists and mathematicians. In General Relativity, trapped surfaces are relevant  to describe those regions of  spacetime characterized by the presence of a black hole. 
In particular, marginally trapped tubes, i.e.\ hypersurfaces foliated by marginally trapped surfaces, may describe, as they do in the Schwarzschild model, the horizon separating the black hole from the rest of  spacetime.
This makes the study of marginally trapped surfaces crucial in determining such horizons.

From the pure geometric viewpoint, there is no reason to restrict ourselves to the case of spacelike surfaces in spacetimes since the definition makes sense in a more general framework: given  a  submanifold $\s$ of a pseudo-Riemannian manifold $({\cal N},g)$, we shall say that $\s$ is \em marginally trapped, \em and use the notation MT,  if $\vec{H}$ is a null vector, i.e.\ $g(\vec{H},\vec{H})$ vanishes identically. Two obvious necessary conditions of  this to happen are that $(i)$  $\s$ has codimension greater than or equal to two and $(ii)$ the induced metric on the normal bundle is indefinite.

The underlying partial differential equation is underdetermined, so we may expect that MT submanifolds are abundant. On the other hand, the fact that the codimension of the surface is greater than one makes  the analysis of the equation difficult. In the context of MT surfaces in spacetimes, a way to overcome this difficulty consists of assuming that the MT surface is contained in a Cauchy hypersurface. This reduces the codimension to one and makes possible the use of the powerful methods of global analysis (\cite{AEM},\cite{Me}).

In \cite{AG}  and \cite{An}, another approach was introduced: using the canonical contact structure of the space of  null geodesics of a given Lorentzian manifold, one can parametrize in a suitable way codimension two sub\-mani\-folds which are orthogonal to a given congruence of null geodesics. This allows us to obtain local representation formulas for MT submanifolds in several  cases: the  Lorentzian space forms $\R^{n+2}_1, d\S^{n+2}$ and $Ad\S^{n+2},$ and the Lorentzian products $\S^{n+1} \times \R$ and $\H^{n+1} \times \R$ (\cite{AG}), as well as the pseudo-Riemannian spaces forms with higher signature (\cite{An}). In the case of surfaces ($n=2$), the formulas are explicit.

The next simplest Lorentzian manifolds are perhaps the Robertson-Walker spaces\footnote{They are also referred as to \em Friedmann-Lema\^itre-Robertson-Walker spaces\em.}. From the geometric viewpoint, these space are the Cartesian product of  the space forms $\R^{n+1}$, $\S^{n+1}$ and $\H^{n+1}$ by a real interval $I$, equipped with a Lorentzian metric which is a warped product.  Robertson-Walker spaces are physically relevant since they describe homogeneous, iso\-tro\-pic expanding/contracting universes.
The purpose of this paper is to address the study of MT submanifolds of Robertson-Walker spaces.

The methods used are analogous to that introduced in \cite{AG} and rely on the use of contact structure of  the
set of null (lightlike) geodesics of the ambient space. The congruence of null lines which are normal to a submanifold of codimension two is a Legendrian submanifold with respect to this contact structure. Conversely, given a null line congruence ${\cal L}$ which is Legendrian, there exists an infinite-dimensional family of submanifolds, parametrized by the set of real maps $\tau \in C^2({\cal L})$,  such that the congruence is normal to them. The function $\tau$ is nothing but the height function of the corresponding submanifold.
In Minkowski space, null geodesics are simply those straight lines supported by a null vector. Since any such straight line crosses exactly once any horizontal hyperplane and since any null vector $\overline{\nu} \in T\R^{n+2}$ may be normalized to the form $(\nu, 1), \nu \in \S^n,$ there is a natural identification between the space of null geodesics of $\R^{n+2}$ with $\R^{n+1}\times \S^n$, the unit tangent bundle of $\R^{n+1}$. If we replace the canonical Lorentzian metric of $\R^{n+2}$ by a warped product, null geodesics are no longer straight lines, but nevertheless their projections on the horizontal hyperplane $\R^{n+1}$ are still  straight lines. It is therefore possible to adapt the method of \cite{AG}: to consider a Legendrian $n$-parameter family of null geodesics amounts, locally, to consider a hypersurface of $\R^{n+1}$. Given such an immersed hypersurface $\varphi: \M \to \R^{n+1}$, it remains to find the condition on the height function $\tau \in C^2(\M)$ insuring that the integral submanifold corresponding to
the pair $(\varphi, \tau)$ is MT. In the flat case treated in \cite{AG}, this condition is polynomial in $\tau$, while in the Robertson-Walker case,   $\tau$ must be solution of a non polynomial algebraic equation (Theorem~\ref{mainT}). It turns out that, making some completeness assumption on the ambient space and using the mean value theorem, we are able to prove the existence of local solutions to this equation (Corollary \ref{sol}). We also discuss the special cases of  curves with null acceleration in three-dimensional spaces Robertson-Walker spaces (Corollary \ref{curves}), as well as that  of  submanifolds with null second fundamental form (Theorem \ref{Tnull2ff}).

\medskip

We thank the referee for improving the first version of this paper, notably pointing out the local nature of Lemma \ref{l2}.

\section{Notations and statement of  results}
Given $c \in \{ -1,0,1\}$, let $\q$ be the  $(n+1)$-dimensional space form of curvature $c$, i.e.\ the unique (up to isometry) simply connected, Riemannian manifold of constant sectional curvature $c$. Explicitly:

\begin{itemize}
\item[---]
 $\mathbb Q^{n+1}_0:=\R^{n+1}$ is the Euclidean space endowed with its canonical metric
$$\<\cdot,\cdot\>_0:= dx_1^2 + \cdots + dx_{n+1}^2;$$
\item[---]
$\mathbb Q^{n+1}_1:=\S^{n+1} = \{x\in\R^{n+2} \, | \, \<x,x\>_0 = 1\}$
is the unit sphere;
\item[---] 
$\mathbb Q^{n+1}_{-1}:=\H^{n+1} = \{x\in\R_1^{n+2} \, | \, \<x,x\>_1 = -1, \, x_{n+2} >0\}$,
where 
$$\<\cdot,\cdot\>_1 := dx_1^2 + \cdots + dx_{n+1}^2 - dx_{n+2}^2,$$ is the hyperbolic space.
\end{itemize}
The metrics of the ambient spaces $\R^{n+2}$ and $\R_1^{n+2}$ induce Riemannian metrics on the sphere and hyperbolic space respectively. For simplicity, we will denote by $\<\cdot,\cdot\>$ both the metrics of $\q$ and of its ambient space.
We define the real functions $\co$ and $\si$ in such a way that the curve $\ga(t) = \co (t) p + \si (t)v$ is the unique geodesic of $\q$ with initial conditions $(\ga(0),\ga'(0)) =( p,v) \in T\q.$ Explicitly, we have:

$$ (\co(t), \si (t)) = \left\{ \begin{array}{ll}  
     (\cos (t),\sin (t)) & \mbox{ if } c=1;\\
  (1,t)  & \mbox{ if } c= 0; \\ 
     (\cosh (t),\sinh (t)) & \mbox{ if } c=-1.
\end{array} \right.$$
Note that $\coca(t) + c \, \sica(t) = 1$
and that the following derivation formulas hold:
$\co(t)' = - c\,  \si(t)$ and
$\si(t)' = \co(t)$.
It is also convenient to introduce the following "cotangent" function:
\begin{eqnarray*}
 \ct(t)=\frac{\co(t)}{\si(t)} := \left\{ \begin{array}{ll}  
                                 \cot (t) & \mbox{ if } c=1;\\
 t^{-1} & \mbox{ if } c=0; \\
                         \coth (t) & \mbox{ if } c=-1.
\end{array} \right.
\end{eqnarray*}
We now introduce the \em Gauss map \em of an immersed hypersurface of $\q$:
let $\M$ be an $n$-dimensional, oriented manifold and $\varphi: \M \to \q$ a smooth immersion. Let $(X_1, \dots,X_n)$ be a local, oriented, tangent frame along $\M$.
The \em Gauss map \em of $\varphi$ is the unique map $\nu $ such that:
\begin{itemize}
\item[$(i)$]
$\<d\varphi(X_i), \nu\>$ vanishes $\forall i, \,  1 \leq i \leq n$;
\item[$(ii)$] the frame $(d\varphi(X_1), \dots, d\varphi(X_n), \nu)$ is positively oriented; 
\item[$(iii)$]$  \<\nu, \nu\>=1.$
\end{itemize}
Observe that if $c=0$ or $c=1$, $\nu$ is $\q$-valued, while in the hyperbolic case $c=-1$,  it takes values in the de Sitter space
$d\S^{n+1} := \{x\in\R_1^{n+2} \, | \, \<x,x\>_1 =~1\}$.

We are now in position to state our  results. 
Let $I$ be an open interval (bounded or not) and $w$ be a smooth real function on $I$ with $w >0.$
We denote by  $\q \times_w I$ the product space $\q \times I$ endowed with the warped metric
$$ \<\cdot,\cdot\>_w := w(t)^2 \<\cdot,\cdot\> - dt^2.$$  It is convenient to introduce the real function $\te(t):= \int_{t_0}^t \frac{ds}{w(s)}$.
A detailed study of the metric properties of $\q \times_w I$ can be found in \cite{ON}.

 We first discuss the case of those codimension two submanifolds of $\q \times~I$ whose second fundamental form $h$ is null:

\begin{theo} \label{Tnull2ff}
  Let $\s$ be a connected, codimension two submanifold of $\q \times_w~I$ with dimension $n>1$ and null second fundamental form. Then:
\begin{itemize}
\item[---] either $\s= {\cal Q}\times \{T\}$, $T\in I$,  where ${\cal Q}$ is a totally umbilic or totally geodesic hypersurface of $\q$ with curvature $w'(T)$;
\item[---] or the function \em $ \Omega(t):=\frac{\te'' \ct(\te)+ c \,  (\te')^2 }{ \te''  -(\te')^2 \ct(\te)}$ \em (which depends only on the warped function $w$) is constant and $\s$
 is locally congruent to an immersion of the form \em
$$\overline{\varphi}=(\co(\te\circ \tau) \varphi + \si(\te\circ\tau) \nu,\tau),$$
\em
where $\varphi :\M \to \q$ is a totally umbilic (or totally geodesic) immersion with Gauss map $\nu$, with curvature $C_0:= \Omega(t)$, and $\tau$ is an arbitrary real map in $ C^2(\M)$.
\end{itemize}
\end{theo}
The next theorem provides a local characterization of codimension two MT submanifolds:

\begin{theo} \label{mainT}
Let $\varphi: \M \to \q$ be an immersed hypersurface of class $C^4$ with Gauss map $\nu$. Denote  by  $ \ka_1 , \ldots, \ka_p, \, p \geq 2$  the $p$ distinct principal curvatures with multiplicity $m_i$  of $\varphi$.
Let $\tau \in C^2(\M)$. Then the map $\overline{\varphi} : \M \to \q \times_w I$ defined
by \em 
$$  \overline{\varphi}:=(\co(\te\circ \tau) \varphi + \si (\te\circ \tau) \nu,\tau)$$
 \em is a MT immersion if and only if \em
\begin{equation} n \, \frac{dw}{dt}\circ \tau - \sum_{i=1}^p m_i  \frac{\ka_i \, \ct(\te \circ \tau)+ c }{\ct(\te \circ \tau)-\ka_i } = 0.
\label{main}
\end{equation}
\em Conversely, any MT, codimension two submanifold of $\q \times_w~I$ is locally congruent to the image of such an immersion. 

\end{theo}

\noindent If the warped factor $w$ is constant, we are in the case of a Cartesian product $ \q \times \R$ and we recover the results of \cite{AG} concerning the Minkowski space $\R^{n+2}_1= \R^{n+1} \times \R$ and the product spaces $\S^{n+1}\times \R$ and $\H^{n+1}\times \R$.  On the other hand, if the image of the immersion $\overline{\varphi}$ is contained  in a time slice $\{ t = T \},$ we recover one of the main results of \cite{FHO}:

\begin{coro}  \label{slice}
A codimension two submanifold $\s={\cal Q}\times \{T\}$ contained in a time slice $\{ t =T \}$   is MT if and only if the hypersurface ${\cal Q} \subset \q$ has constant mean curvature $w'(T)$.
\end{coro}

In general, Equation (\ref{main}) cannot be solved explicitely. However, if $n=1$, i.e.\ for curves with null acceleration, a more precise statement may be given:

\begin{coro} \label{curves}
A regular, non null curve of $\mathbb Q^2_c \times_w~I$ with null acceleration takes the form
\em
$$\overline{\ga}=(\co(\te\circ \tau) \ga+ \si(\te\circ\tau) \nu,\tau),$$
\em
where $\ga$ is a regular curve of $\mathbb Q^2_c$ with unit normal vector $\nu$, $\tau \in C^2(\R)$ and the curvature $\ka$ of $\ga$ with respect to $\nu$ satisfies: \em
$$ \ka = \left( \frac{w'  \ct(\te)- c  }{w' + \ct (\te)  }\right) \circ \tau.$$ \em
\end{coro}

In higher dimension, under a natural assumption of geodesic completeness, we are able to prove the existence of local solutions to Equation ({\ref{main}), in terms of an arbitrary hypersurface of $\q$:

\begin{coro} \label{sol}
Assume that $\q \times_w I$ is future- or past-null geodesically complete, i.e.\ $\int_I \frac{dt}{w(t)}=+ \infty$.
Let $\varphi: \M \to \q$ be an immersed hypersurface of class $C^4$ with Gauss map $\nu$, whose principal curvatures $\ka_i$, $1 \leq i \leq p$, have constant multiplicity.
Let $q$ be the integer number defined by
\begin{eqnarray*}
 q := \left\{ \begin{array}{ll}  
                        p-1 \quad & \mbox{ if } c=1;\\
  \#\{ \ka_i | \, \,  \ka_i \neq 0\} -1&  \mbox{ if } c=0;\\
 \#\{ \ka_i | \,  \,  |\ka_i| > 1\}-2  & \mbox{ if } c=-1.
\end{array} \right.
\end{eqnarray*}
Then for all point $ x$  in  $\M$, 
there exists a neighbourhood $U$ of $x$ in $\M$ and $q$ solutions $\tau_i \in C^2(U)$ of
Equation (\ref{main}), giving rise to $q$ MT immersions $\overline{\varphi}_i : U \to \q \times_w I$ of the form \em
$$ \overline{\varphi}_i = (\co (\te\circ\tau_i) \varphi + \si (\te\circ\tau_i) \nu , \tau_i).$$
\em
\end{coro}

\section{The local geometry of codimension two, spacelike submanifolds of $\q \times_w I$}

Let $\overline{\varphi}=(\psi,\tau)$ be an immersion of a $n$-dimensional manifold $\M$ into $ \q \times_w~I $ which is spacelike,
i.e.\ the induced metric $\overline{g}:=\overline{\varphi}^*\<\cdot,\cdot\>_w$ is definite positive.

We denote by $\overline{\nu}$ one of the two globally defined null normal vector fields along $\overline{\varphi}$, normalized   as follows: $ \overline{\nu}=(\chi,w \circ\tau),$ with $\<\chi ,\chi \>=1$.
In particular $\<d\overline{\varphi},\overline{\nu}\>_w=0$, i.e.\
$$w^2 \<d\psi, \chi\> -w \circ \tau \,  d\tau=0,$$
which implies $\<d\psi,\chi\> =  w^{-1}d\tau.$

We introduce  the maps
\begin{eqnarray*}
\varphi &:=& \co (\te\circ\tau) \psi - \si(\te\circ\tau) \chi \label{definitions}\\
\nu &:=&c\,  \si(\te\circ\tau) \psi + \co (\te\circ\tau) \chi,
\end{eqnarray*}
where $\te$ is a function of the real variable $t$ that we shall determine later. Observe for the moment that 
if $c \neq 0$,  $\<\varphi, \nu\>$ vanishes. Hence, in all cases, $\nu \in T_\varphi \q.$
The following formulas will be useful later:
\begin{eqnarray*}
 \psi &=& \co(\te\circ\tau) \varphi + \si(\te\circ\tau) \nu \\
 \chi &=& -c \, \si(\te\circ\tau) \varphi + \co(\te\circ\tau) \nu
\end{eqnarray*}
and
\begin{equation}\label{eq:d-formulas}
\begin{split}
 d\varphi &= -\te' d\tau \nu + \co(\te) d\psi - \si(\te) d\chi  \\
 d\nu &= c \, \te' d\tau \, \varphi + c \, \si(\te) d\psi + \co(\te) d\chi  \\
 d\psi &= \te' d\tau \, \chi + \co(\te) d\varphi + \si(\te) d\nu.
\end{split}
\end{equation}

We now choose $\te$ in such a way that $\nu$ is the Gauss map of $\varphi,$ when $\varphi$ is an immersion.

In the following, the sign $'$ denotes the differential $\frac{d}{dt}$ on the real line $\R$
and when there is no risk of confusion, the composition $\te\circ\tau$ will be simply denoted by $\te$.

Taking into account that $d\varphi=-\te'd\tau \, \nu + \co(\te)d\psi - \si(\te)d\chi$,
we calculate
\begin{eqnarray*}
  \< d\varphi,\nu \> &=& -\te' d\tau \< \nu, \nu\> + \, \< \co(\te)d\psi - \si(\te)d\chi,  c\,  \si(\te) \psi + \co (\te) \chi \> \\
 &=& -\te' d\tau + \coca(\te) \<d\psi, \chi\> + c\, \co(\te)\si(\te) \<d\psi, \psi\> \\
 && \, \, -c \, \sica(\te) \<d\chi,\psi\>.
\end{eqnarray*}
In the flat case $c=0$, we get  $\<d\varphi, \nu\> = -\te' d\tau +  \<d\psi, \chi\>$. We claim that the same holds when $c\neq0$: first, since $\<\psi,\psi\>$ is constant, $\<d\psi,\psi\>$ vanishes. On the other hand,  
$\chi$ is tangent to $\q$ at $\psi$, so that $\<d\psi,\chi\>+\<\psi, d\chi\>$ vanishes as well. Hence we obtain
$$\< d\varphi,\nu \>=-\te' d\tau + ( \coca(\te) + c\, \sica(\te) ) \<d\psi, \chi\>.$$
  Hence we set $\te(t):= \int_{t_0}^t \frac{ds}{w(s)}$ and we deduce that $\<d\varphi,\nu\>=0$.

%\begin{equation}\label{eq:d-formulas}
%\begin{split}
 %d\varphi &= -\te' d\tau \nu + \co(\te) d\psi + \si(\te) d\chi  \\
 %d\nu &= c \, \te' d\tau \, \varphi + c \, \si(\te) d\psi + \co(\te) d\chi  \\
 %d\psi &= \te' d\tau \, \chi + \co(\te) d\varphi + \si(\te) d\nu
%\end{split}
%\end{equation}
%and
%$$\< d\varphi,\chi \> = \< d\nu,\chi \> = 0.$$

\begin{lemm}
 The map $(\varphi, \nu) : \M \to  T \q$ is an immersion. 
\end{lemm}

\begin{proof}
Suppose $(\varphi,\nu)$ is not an immersion, so that there exists a non-vanishing vector $v \in T\M$ such that $(d\varphi(v),d\nu(v))=(0,0).$
Using Equation \eqref{eq:d-formulas} we have:
\begin{eqnarray*}
   d \overline{\varphi}(v) &=&(d\psi(v), d\tau(v)) \\
&=& (\te' d\tau(v) \chi + \co(\te) d\varphi(v) + \si(\te) d\nu(v) , d\tau(v))\\
&=&  (\te' \chi, 1) d\tau(v) .
\end{eqnarray*}
Observe that $(\te' \chi, 1)$ is a null vector, so $d \overline{\varphi}(v)$ as well. This  contradicts the assumption that $\overline{\varphi}$ is spacelike.
\end{proof}

\begin{lemm} \label{l2} Given $x \in \M$ and $ \eps >0,$ there exists a neighbourhood $U$ of $x$ and $t_0 \in (-\eps, \eps)$ such that \em$\co(t_0) \varphi + \si(t_0) \nu$ \em is an immersion of $U$, and
 \em $\co(t_0) \nu + c \,  \si(t_0) \varphi$ \em is its Gauss map\footnote{This corresponds to the fact that the immersion $(\varphi,\nu)$ is \em Legendrian \em with respect to the canonical contact structure of the unit tangent bundle of $\q$.}.
  \end{lemm}
 
\begin{proof}
Let $(t,t')$ be a pair of  real numbers  such that $t \neq t'$, if $c=-1$ or $c=0$, or such that $t -t' \neq 0 \, [\pi]$, if $c=1$.
Since the pair $(\varphi, \nu)$ is an immersion, we have, at a given point $x$:
$$Ker (\co(t) d\varphi + \si(t) d \nu) \cap Ker  (\co(t') d\varphi +\si( t')  d \nu) = \{ 0\}.$$
It follows that the set
 $ \{ t \in \R \, | \,  \,   \co(t) d\varphi_x +  \si(t)  d\nu_x \mbox {  has not  maximal rank} \}$ contains at most $n$ elements (the dimension of ${\cal M}$) so its complementary is dense in $\R.$ It follows that there exists a  neighbourhood $U$ of $x$ such that 
 $ \{ t \in \R \, | \,  \,   \co(t) \varphi +  \si(t)  \nu \mbox { is  an immersion of }U \}$ contains a neighbourhood of $0$, which implies the first part of the claim.

We then calculate
\begin{eqnarray*} \<d( \co(t_0)  \varphi +  \si(t_0)  \nu),\co(t_0) \nu + c \,  \si(t_0) \varphi\>_0&&\\
 \hspace{6em} =
c \, \si(t_0) \left(  \co(t_0) \< d\varphi, \varphi\> + \si(t_0) \< d\nu, \varphi\>\right).
\end{eqnarray*}
(we have used the fact that both $\<d\varphi, \nu\>$ and $\<d\nu, \nu\>$ vanish). If $c \neq 0,$ $\< d\varphi, \varphi\>$ vanishes as well since $\<\varphi, \varphi\>$ is constant, and 
$ \< d\nu, \varphi\>= -  \< d\varphi, \nu\>$ vanishes as well. 
Hence in all cases $\co(t_0) \nu + c \, \si(t_0) \varphi$ is the Gauss map of $ \co(t_0) \varphi + \si(t_0) \nu$.\end{proof}

Observe that applying a vertical translation $(\psi, \tau) \mapsto (\psi, \tau+t_0)$ to $\overline{\varphi}$, with the constant $t_0$ of Lemma \ref{l2}, has the effect of replacing the map $\varphi$  by $\co(t_0) \varphi + \si(t_0) \nu$. Since the whole discussion is local, there is no loss of generality in assuming that $\varphi$ is an immersion, which we will do henceforth.

In order to compute the geometry of the immersion $\overline{\varphi}$, we will make use of the following lemma, proved in \cite{ON}:
\begin{lemm}\label{lemm:oneill}
Let $X,Y$ be two vector fields on $\q$ and denote by $\pa_t$ the canonical vector field on $I$.
Then the Levi-Civita connection $\overline{D}$ of $\<.,.\>_w$ is related to the Levi-Civita connection $D$ of $\<.,.\>$ by the following relations:
\begin{eqnarray*}
&(i)& \overline{D}_{(0,\pa_t)} (0,\pa_t) = 0\\
&(ii)& \overline{D}_{(0,\pa_t)} (X,0) =  \overline{D}_{(X,0)} (0,\pa_t) = w^{-1} w' (X,0)\\
&(iii)& \overline{D}_{(X,0)} (Y,0) = w w' \<X,Y\> (0,\pa_t) + \big( D_X Y, 0 \big).
\end{eqnarray*}
\end{lemm}

\begin{lemm}\label{lemm:metric}
Denote by $g:=\varphi^*\<\cdot,\cdot\>$ the metric induced on $\M$ by $\varphi$ and $A$ the shape operator associated to $\nu$.
Then the metric $\overline{g}=\overline{\varphi}^* \<\cdot,\cdot\>_w$ induced on $\M$ by $\overline{\varphi}$ is given by \em 
 \begin{eqnarray*}
  \overline{g}= w^2 \Big( \coca(\te)\,  g - 2 \si(\te) \co(\te) g(A \, \cdot, \cdot) + \sica(\te) g(A \, \cdot, A \, \cdot) \Big).
 \end{eqnarray*} \em
Let $(e_1, \ldots,e_n)$ be a local frame on $\M$ which is $g$-orthonormal, i.e.\ $g(e_i,e_j)=~\delta_{ij}$, and principal, i.e.\ it diagonalizes $A$.
 We set
  $\overline{E}_i:=d\overline{\varphi}(e_i), \, \, \forall  i=1,\ldots,n.$ Then the second fundamental form of $\overline{\varphi}$ in the direction of the null vector $\overline{\nu}$ is given by:
\em
\begin{eqnarray*} 
  \< \overline{D}_{\overline{E}_i} \overline{\nu} , \overline{E}_j \>_w &=& 
  -w^2(\co(\te) - \ka_i \si(\te)) ( \ka_i \co(\te) + c \, \si (\te))  \delta_{ij}\\
&&  + w^2 w' (\co(\te) - \ka_i \si(\te) )^2  \delta_{ij}.
 \end{eqnarray*}
\em
\end{lemm}
\begin{rema} \em If $w=1$ we recover the formulas derived in \cite{AG} in the cases $\R^{n+2}_1, \S^{n+1}\times \R$ and $\H^{n+1}\times \R$. \em \end{rema}

\begin{proof}
Using  Equation \eqref{eq:d-formulas} we calculate
 \begin{eqnarray*}
  \< d\psi,d\psi \> &=& 
          \coca(\te) \< d\varphi,d\varphi \> + 2 \si(\te) \co(\te) \< d\varphi,d\nu \> \\
                               && + \, \sica(\te) \< d\nu,d\nu \> + (\te')^2 d\tau^2\\
                   && + 2 \te' d\tau \Big( \co(\te) \< d\varphi,\chi \> + \si(\te) \< d\nu,\chi \> \Big) \\
                      &=& \coca(\te) \< d\varphi,d\varphi \> + 2 \si(\te) \co(\te) \< d\varphi,d\nu \> \\
                               && + \, \sica(\te) \< d\nu,d\nu \> + (\te')^2 d\tau^2.
 \end{eqnarray*}
 It follows that
 \begin{eqnarray*}
  \< d\overline{\varphi} , d\overline{\varphi} \>_w &=& w^2 \< d\psi,d\psi \> - d\tau d\tau \\
          &=& w^2 \Big( \coca(\te) \< d\varphi,d\varphi \> + 2 \si(\te) \co(\te) \< d\varphi,d\nu \>\\ 
               &&   + \, \sica(\te) \< d\nu,d\nu \> \Big).
 \end{eqnarray*}
 Taking into account that $-d\nu =d\varphi \circ A,$ this gives the required  formula for~$\overline{g}$.
 
\medskip

We now calculate the second fundamental form of $\overline{\varphi}$. Setting $E_i=d\varphi(e_i)$, we first observe that
\begin{eqnarray}\label{eq:dpsi}
  d\psi(e_i) = (\co(\te)-\ka_i \si (\te))E_i +w^{-1} d\tau(e_i)\chi.
 \end{eqnarray}
 Splitting the covariant derivative $\overline{D}_{\overline{E}_i} \overline{\nu}$ in four terms and applying Lemma~\ref{lemm:oneill} to each of them, we get:
 \begin{eqnarray*}
  \overline{D}_{\overline{E}_i} \overline{\nu}&=& \overline{D}_{(d\psi(e_i),0)} (\chi,0) + \overline{D}_{(d\psi(e_i),0)} (0, w ) \\
                      &&  +   \overline{D}_{(0,d\tau(e_i))} (\chi,0) + \overline{D}_{(0,d\tau(e_i))} (0, w ) \\
                           &=& ww' \< d\psi(e_i),\chi \> (0,1) + ( D_{d\psi(e_i)}\chi,0 ) +w'(d\psi(e_i),0) \\
                                  && +  d\tau(e_i) w'w^{-1} (\chi,0) + d\tau(e_i)w'(0,1) \\
           &=& w' d\tau(e_i)(0,1)+   w'(d\psi(e_i),0)+\Big( D_{d\psi(e_i)} (-c \, \si (\te) \varphi + \co(\te) \nu), 0 \Big)\\
                               && + d\tau(e_i) w' w^{-1}( \chi, w)\\
                                         &=& w' \overline{E}_i + \Big( -c\si (\te) E_i-  \co(\te) \ka_iE_i , 0 \Big)  + d\tau(e_i)w'w^{-1}\overline{\nu}.
  \end{eqnarray*}
(we have use the fact that $\<d\psi, \chi\>=w^{-1}d\tau$).
Recalling that $\overline{\nu}$ is normal to $\overline{\varphi}$, we conclude
\begin{eqnarray*}
   \< \overline{D}_{\overline{E}_i} \overline{\nu}, \overline{E}_j\>_w&=&w'  \< \overline{E}_i, \overline{E}_j\>_w- (  \ka_i \co(\te) +c\si (\te)) \< (E_i , 0),  (d\psi(e_j), d\tau(e_j) \>_w
   \\
&=&  + w^2 w'  (\co(\te)- \ka_i \si (\te))^2 \delta_{ij}\\
 &&-w^2  (\ka_i \co(\te) +c\,\si (\te)  ) \< E_i,  (\co(\te)- \ka_j \si (\te))E_j \> \\
  && -w^2  (  \ka_i \co(\te)+c\,\si (\te) ) \< E_i, w^{-1} d\tau(e_j)\chi \>\\
  &=& w^2 w'  (\co(\te)- \ka_i \si (\te))^2 \delta_{ij}\\
&& -w^2  (  \ka_i \co(\te)+ c\,\si (\te) )  (\co(\te)-\ka_i\si (\te))\delta_{ij}.
  \end{eqnarray*}
 \end{proof}

%%%%%%%%%%%%%%%%%%%%%%%%%%%%%%%%%%%%%%%%%%%%%%%%%%%%%%%
\section{Proof of the results}

\subsection{Proof of Theorem \ref{mainT}}

Let $\overline{\varphi}$ be a spacelike, codimension two immersion of $\q \times_w~I$ and denote by $\vec{H}$ its mean curvature vector field. Since the normal plane is two-dimensional, the MT assumption $\<\vec{H},\vec{H}\>_w=0$ is equivalent to the fact that $\vec{H}$ is collinear to one of the two normal, null fields. Since the choice of $\overline{\nu}$ in the previous section was arbitrary, there is no loss of generality in studying the case in which $\vec{H}$ is collinear to $\overline{\nu}$, i.e.\  $\< \vec{H},\overline{\nu} \>_w$ vanishes.

We use the notations and the results of  Lemma \ref{lemm:metric}. In particular, recalling that the local frame $(\overline{E}_1, \dots,\overline{E}_n)$ is orthogonal (but not orthonormal), we have:
\begin{eqnarray*}
 \< \vec{H},\overline{\nu} \>_w &=& \frac{1}{n} \sum_{i=1}^n \frac{-\< \overline{D}_{\overline{E}_i} \overline{\nu} , \overline{E}_i\>_w }{\<\overline{E}_i,\overline{E}_i\>_w}\\               
                     &=& \frac{1}{n} \sum_{i=1}^n \Big( \frac{\ka_i \co(\te) + c \, \si(\te)}{\co(\te) - \ka_i \si(\te)} - w' \Big) \\
                     &=& \frac{1}{n} \sum_{i=1}^n \Big( \frac{\ka_i  \ct(\te)+ c }{\ct(\te) - \ka_i } \Big) - w'.
\end{eqnarray*}
Therefore, counting the multiplicities of the principal curvatures, $\< \vec{H}, \overline{\nu} \>_w $ vanishes if and only if
\begin{eqnarray*}
 n\, w' - \sum_{i=1}^p m_i  \frac{\ka_i  \ct(\te)+ c }{\ct(\te) - \ka_i } = 0,
\end{eqnarray*}
which is the required formula.

\subsection{Proof of Corollary \ref{slice}}

The fact that $\overline{\varphi}$ is contained in a time slice means that $\tau$ is constant, and therefore $\te \circ \tau $ and $w' \circ \tau$ are constant as well. In particular, by Equation (\ref{main}), the quantity
$$\sum_{i=1}^p m_i  \frac{\ka_i  \ct(\te)+ c }{\ct(\te) - \ka_i }$$
 is constant.
Moreover, for $\te$ small enough in order to avoid focal sets, the image of the immersion $\psi= \co(\te) \varphi+ \si(\te) \nu$ is nothing but the equidistant hypersurface to  $\varphi(\M)$ in $\q$ at distance $\te.$  It is easy to check that its principal curvatures are 
$ \frac{\ka_i  \ct(\te)+ c }{\ct(\te) - \ka_i }$. Therefore, the mean curvature of $\psi$ in $\q$ is the constant
$\frac{1}{n} \sum_{i=1}^p m_i  \frac{\ka_i  \ct(\te)+ c }{\ct(\te) - \ka_i },$ which, by Equation  (\ref{main}) is equal to $w'(\tau)$.

%%%%%%%%%%%%%%%%%%%%%%%%%%%%%%%%%%%%%%%%%%%%%%%%%%%%%%%%%%%%%%%%%%%%%%%%%%%%%%%%
\subsection{Curves with null acceleration (proof of Corollary \ref{curves})}
 Writing 
$\overline{\ga}$ (resp.\ $\ga$) instead of $\overline{\varphi}$ (resp.\ $\varphi$), we have
 $\overline{\ga} := (\cos (\te \circ \tau) \ga+\sin (\te \circ \tau) \nu , \tau)$,
where $\ga$ is  a parametrized curve of $\mathbb Q^2_c$ with unit normal $\nu$.
Equation (\ref{main}), which is now the necessary and sufficient condition for the  acceleration of the curve  $\overline{\ga}$ to be collinear to the null vector $\overline{\nu}= (\chi, w \circ \tau)$, becomes
$$  w' \circ \tau-\frac{\ka \, \ct (\te \circ \tau)+ c }{\ct  (\te\circ \tau)- \ka} =0,$$
where $\ka$ is the curvature of $\ga$. This is equivalent to:
$$ \ka = \left( \frac{w'  \ct(\te)- c  }{w' + \ct (\te)  }\right) \circ \tau.$$
Hence the curve $\ga$ is explicitely determined, via its curvature function,  by the warped factor $w$ and the height function $\tau$. 
Moreover, if the function $\Omega(t):=\frac{w'  \ct(\te)- c  }{w' + \ct (\te)  }$ is  invertible, given an arbitrary curve $\ga$ of $\mathbb Q^2_c$  with curvature function $\ka$, we can build a curve $\overline{\ga} $ of  $\mathbb Q^2_c \times_w I $ with null acceleration by setting $\tau := \Omega^{-1} \circ \ka$ (here $\Omega^{-1}$ denotes the inverse with respect to the composition, not the algebraic inverse).

\medskip

An interesting case in which the function $\Omega$ is constant is  when $w(t)=e^t$ and $c=0$, since 
$\R^2 \times_{e^t} \R$ is locally isometric to a dense open subset of the de Sitter space $d\S^3$.
In this case, since $w^{-1}(t) = e^{-t}$ and
$\te (t)= -e^{-t} + C$, we deduce that
$$\ka= \frac{e^{\tau}}{(-e^{-\tau} + C )e^{\tau}+1}=\frac{e^\tau}{C e^{\tau}}=C^{-1}.$$
Hence here $\ga$ must be a circle. Since the situation is similar to that of codimension two submanifolds with null second fundamental form, we refer to Section \ref{null2ffS} for more detail.

%%%%%%%%%%%%%%%%%%%%%%%%%%%%%%%%%%%%%%%%%%%%%%%%%%%%%%%%%%%%%%%%%%%%%%%%%%%%%%%%

\subsection{Local existence of solutions (proof of Corollary \ref{sol})}
Let $(\ka_i, \ka_{i+1})$ be a pair of consecutive principal curvatures, i.e.\ there is no other principal curvature in the open interval $(\ka_i,\ka_{i+1})$. If $c \neq 1$, we need one more assumption: if
$c=0$, we
furthermore assume that $\ka_i \ka_{i+1} \neq 0$, while if $c=-1$, we make the  assumption that either
$\ka_i, \ka_{i+1}\ > 1$ or $\ka_i, \ka_{i+1}\ <-1$. In all cases, it follows that there exists a real interval  $[\te_{1}^i, \te_2^i] $ such that the real function $\ct$ is a bijection from $[\ka_i,\ka_{i+1}]$ into $[\te_1^i,\te_2^i]$. %In particular
%$ \ct(\te_1^i)= \min(\ka_i,\ka_{i+1})$ and $\ct(\te_2^i) =\max(\ka_i, \ka_{i+1})$.
We now introduce:
$$G(s):=-n \, w' \circ \, \te^{-1} \prod_{k=1}^p (\ct(s)- \ka_k ) + \sum_{k=1}^p m_k (\ct(s) \ka_k + c )  \prod_{j\neq k}^p (\ct(s) - \ka_j ).$$
A calculation shows that
\begin{eqnarray*} G(\te_1^i)G(\te_2^i) &=& m_i m_{i+1} (\ka_i^2 +c )  (\ka_{i+1}^2 +c ) \prod_{j \neq i}^p (\ka_i-\ka_j) \! \! \!\prod_{j \neq i+1}^p \! \!(\ka_{i+1}-\ka_j)\\
 &=&   m_i m_{i+1} (\ka_i^2 +c )  (\ka_{i+1}^2 +c )  (\ka_i-\ka_{i+1})(\ka_{i+1}-\ka_i) \!\!\! \!\prod_{j \neq i, i+1}^p \! \! \!(\ka_i-\ka_j)^2 \\
 &< &0.
\end{eqnarray*}
Hence, by the mean value theorem, for any point $x \in \M$ such that the assumptions made on $\ka_i$ and $\ka_{i+1}$ hold, there exists $s_i(x) \in (\te_1^i,\te_2^i)$ such that $ G(s_i) $ vanishes. Since the real function map $G$ is of class $C^2$, the function $s_i(x)$ is $C^2$ as well by the implicit function theorem. 
 Recalling  that  $\int_I \frac{dt}{w(t)}=+\infty$, i.e.\  $\te$ is a bijection from $I$ to $\R$, we set $\tau_i:= \te^{-1} \circ s_i$. Clearly, $\tau_i$ is a solution of (\ref{main}) and is $I$-valued.

To conclude the proof, observe that in the elliptic case $c=1$, we have obtained exactly $p-1$ solutions, while in the flat case $c=0$, the number of solutions is equal to the number of non-vanishing principal curvatures minus one. Finally, in the hyperbolic case $c=-1$, the number of solutions is equal to the number of  principal curvatures $\ka_i >1$ minus one, plus the number of  principal curvatures $\ka_i <1$ minus one.

%%%%%%%%%%%%%%%%%%%%%%%%%%%%%%%%%%%%%%%%%%%%%%%%%%%%%%%%%%%%%%%%%%%%%%%%%%%%%%%%
\subsection{Submanifolds with null second fundamental form (proof of Theorem \ref{Tnull2ff})}
\label{null2ffS}
Lemma \ref{lemm:metric} implies that 
the assumption "null second fundamental form" amounts to the vanishing of 
$$ \ka_i \co (\te)+ c \,  \si(\te) - w' (\co(\te) - \ka_i \si(\te)), \quad \forall i, \, 1 \leq i \leq p .$$
 This is equivalent to
\begin{equation} \ka_i = \left( \frac{w' \,  \ct(\te)- c }{w' + \ct(\te)  } \right) \circ \tau=:\Omega \circ \tau, \quad  \forall  i, \, 1 \leq i \leq p.  \label{null2ff}
\end{equation}

We first deduce from this equation that  the immersion $\varphi$ cannot have two distinct principal curvatures at a given point $ x\in \M$, i.e.\  $\varphi$ is umbilical (or totally geodesic). It is well known that, if $n>1$, 
umbilical hypersurfaces in space forms have constant curvature. % and are therefore open subsets of round hyperspheres or hyperplaines or ...

Next, since now the curvature does not depend on the point $x \in \M$ then either $\tau$ or $\Omega$ must be constant. If $\Omega$ is not constant, $\tau$ must be constant, i.e.\ $\overline{\varphi}$ is contained in a time slice $\{t = T\}.$
This is a special case of Corollary \ref{slice}. Hence the hypersurface $\psi(\M)$ is equidistant to $\varphi(\M)$, and is totally umbilic (or totally geodesic) as well. According to Equation (\ref{main}), its curvature is then the constant $w'(\tau)$. We then get the first part of Theorem \ref{Tnull2ff}.

If $\Omega$ is constant, $\tau$ may be chosen arbitrary and we get the second part of the theorem.

\subsubsection*{Recovering the de Sitter space}

We end this section by a short discussion about a special case in which $\Omega(t)$ is constant.
Recalling that $\te'=w^{-1}$, we deduce that $w'= -\frac{\te''}{(\te')^2}.$
Hence the assumption $\Omega(t)$ is constant is equivalent to the second order ODE:
$$\frac{\te''   \ct(\te)+ c  (\te')^2}{ \te''  -(\te')^2 \ct(\te)   }= C_0.$$
It seems difficult to solve explicitely  this equation, but there does exist a solution which is interesting geometrically, namely $c=0$ and $w(t)=e^t$: the warped product 
$\R^{n+1} \times_{e^t} \R$ is actually isometric to a dense open subset of the de Sitter space $d\S^{n+2}$. This isometry can be realized by setting
$$\left\{ \begin{array}{cclc}  x_i &:=& e^t y_i  & \quad \quad \quad 1 \leq i \leq n+1,\\  
 x_{n+2} &:=& \cosh(t) - \frac{e^t}{2}\sum_{i=1}^{n+1} y_i^2&\\ 
                             x_{n+3} &:=& \sinh(t) + \frac{ e^t}{2} \sum_{i=1}^{n+1} y_i^2&
\end{array} \right.$$
where $(y_1, \ldots, y_{n+1})$ denote the canonical coordinates on $\R^{n+1}$ and $(x_1, \ldots , x_{n+3})$ are canonical coordinates on $\R^{n+3}$. Here we  use the hyperboloid model
$$ d\S^{n+2} = \{  (x_1, \ldots , x_{n+3} ) \in \R^{n+3} \, | \, \,  x_1^2 + \cdots + x_{n+1}^2 + x_{n+2}^2 - x_{n+3}^2=1\}.$$
It is proved in \cite{AG} that a submanifold of $d\S^{n+2}$ of codimension two with null second fundamental form can be locally parametrized by
$$ \overline{\varphi}_1 (x) = (\iota(x),0)+  \tau_1(x) (\chi_1,1),$$
where $\iota:  \S^n \to \S^{n+1}$ is a totally geodesic embedding with Gauss map  $\chi_1$ (this is a constant vector), and $\tau_1 \in C^2(\S^{n})$; for example we can take $\iota(x):= (x,0)$, and  $\chi_1:=(0,1)$.
By a straightforward calculation,  the composition of  $ \overline{\varphi}_1$ with the isometry $d\S^{n+2} \to \R^{n+1} \times_{e^t} I$ gives the immersion
$$\overline{\varphi}_2 = (e^{-\tau_2}x, \tau_2),$$
where  $\tau_2 := \log(2\tau_1).$
We now see that this is consistent with Theoreom \ref{Tnull2ff}}:
we have, $\forall (x,v) \in T\S^n,$
$$ d\overline{\varphi}_2(v) = (  (v- d\tau_2(v) x) e^{-\tau_2}, d\tau_2(v))$$
so that a null, normal vector along $\overline{\varphi}_2$ is
$$\overline{\nu}(x)= (-x, e^{\tau_2}).$$
Hence, using the fact that here $\theta(t)= -e^{-t} + C,$ where $C$ is a real constant, we obtain
$$\varphi_2 = \psi - \theta \chi = e^{-\tau_2}x -(-e^{-\tau_2} + C) (-x) =Cx.$$
Hence $\varphi_2$ is a totally umbilic immersion, as required.

%\begin{eqnarray*} G(\te_i) &=& m_i(\ka_i^2 +c ) \prod_{j \neq i}^p (\ka_i-\ka_j)\\
%&=&m_i\frac{\ka_i^2+c}{\ka_i^{p}} \prod_{j \neq i}^p (\ka_i -\ka_j).
%\end{eqnarray*}
%Let $(\ka_i, \ka_{i+1})$ be a pair of principal curvatures such that there is no other principal curvature in the open interval $[\ka_i,\ka_{i+1}]$.
%Hence the map $G$ is well defined in $[\te_i,\te_{i+1}]$. Since the signs of $G(\te_i)$ and $G(\te_{i+1})$ are opposite, there exists 
%$s_i$ ....
%**********************

%Calculer $G(\te_i)=m_i (\ka_i^2+c) \prod_{j \neq i}^p (\ka_i -\ka_j)$ avec $\ct(\te_i)=\ka_i$.

%Si $c \neq 1$, calculer $\lim_{s \to \pm \infty} G(s)= n(-1)^{p+1} \prod_j \ka_j( w'\circ \te+1) ???$

%%%%%%%%%%%%%%%%%%%%%%%%%%%%%%%%%%%%%%%%%%%%

%%%%%%%%%%%%%%%%%%%%%%%%%%%%%%%%%%%%%%%%%%%%%%%%%%%%%%%%%%%%

\medskip

\noindent

\begin{multicols}{2}
\noindent   Henri Anciaux \\
Universit\'e Libre de Bruxelles \\
  CP 216, local O.7.110  \\
   Bd du Triomphe \\
 1050 Brussels, Belgium   \\
henri.anciaux@gmail.com \\

\columnbreak

\noindent Nastassja Cipriani \\
KU Leuven, Geometry Section \\
Celestijnenlaan 200B \\
3001 Leuven, Belgium \\
nastassja.cipriani@wis.kuleuven.be

\end{multicols}

\end{document}